\documentclass[10pt]{amsart}
\usepackage{epsfig,amsfonts,amsthm,amssymb,latexsym,amsmath,enumitem}

\newcommand{\mymod}[3]{#1 \equiv #2 \kern -0.5em \pmod{#3}}
\newcommand{\mynotmod}[3]{#1 \not \equiv #2 \kern -0.6em \pmod{#3}}

\theoremstyle{plain}
\newtheorem{theorem}{Theorem}[section]
\newtheorem{corollary}[theorem]{Corollary}

\theoremstyle{remark}

\theoremstyle{definition}

\title{Third-order Jacobsthal Generalized Quaternions}
\vspace{5pt}

\author[G. Cerda-Morales]{\scriptsize GAMALIEL CERDA-MORALES$^{1}$}
\date{}

\begin{document}
\maketitle

\vspace{-20pt}
\begin{center}
{\footnotesize $^1$Instituto de Matem\'aticas, Pontificia Universidad Cat\'olica de Valpara\'iso, \\
Blanco Viel 596, Valpara\'iso, Chile.\\
E-mails: gamaliel.cerda@usm.cl / gamaliel.cerda.m@mail.pucv.cl
}\end{center}

\hrule

\begin{abstract}
In this paper, the third-order Jacobsthal generalized quaternions are introduced. We use the well-known identities related to the third-order Jacobsthal and third-order Jacobsthal-Lucas numbers to obtain the relations regarding these quaternions. Furthermore, the third-order Jacobsthal generalized quaternions are classified by considering the special cases of quaternionic units. We derive the relations between third-order Jacobsthal and third-order Jacobsthal-Lucas generalized quaternions.
\end{abstract}

\medskip
\noindent

\medskip
\noindent
\medskip

\hrule

\section{Introduction and Preliminaries}\label{sec:1}
\setcounter{equation}{0}

Recently, the topic of number sequences in real normed division algebras has attracted the attention of several researchers. It is worth noticing that there are exactly four real normed division algebras: real numbers ($\Bbb{R}$), complex numbers ($\Bbb{C}$), quaternions ($\Bbb{H}$) and octonions ($\Bbb{O}$). In \cite{Bae} Baez gives a comprehensive discussion of these algebras.

The real quaternion algebra $$\Bbb{H}=\{q=q_{r}+q_{i}\textbf{i}+q_{j}\textbf{j}+q_{k}\textbf{k}:\ q_{r},q_{s}\in \Bbb{R},\ s=i,j,k\}$$ is a 4-dimensional $\Bbb{R}$-vector space with basis $\{\textbf{1}\simeq e_{0},\textbf{i}\simeq e_{1},\textbf{j}\simeq e_{2},\textbf{k}\simeq e_{3}\}$ satisfying multiplication rules $q_{r}\textbf{1}=q_{r}$, $e_{1}e_{2}=-e_{2}e_{1}=e_{3}$, $e_{2}e_{3}=-e_{3}e_{2}=e_{1}$ and $e_{3}e_{1}=-e_{1}e_{3}=e_{2}$.

There has been an increasing interest on quaternions and octonions that play an important role in various areas such as computer sciences, physics, differential geometry, quantum physics, signal, color image processing and geostatics (for more details, see \cite{Ad,Car,Go1,Go2,Ko1,Ko2}).

A variety of new results on Fibonacci-like quaternion and octonion numbers can be found in several papers \cite{Cer,Cim1,Cim2,Hal1,Hal2,Hor1,Hor2,Iye,Ke-Ak,Szy-Wl}. The origin of the topic of number sequences in division algebra can be traced back to the works by Horadam in \cite{Hor1} and by Iyer in \cite{Iye}. In this sense, Horadam \cite{Hor1} defined the quaternions with the classic Fibonacci and Lucas number components as
\[
QF_{n}=F_{n}+F_{n+1}\textbf{i}+F_{n+2}\textbf{j}+F_{n+3}\textbf{k}\ \ (F_{n}\textbf{1}=F_{n})
\]
and
\[
QL_{n}=L_{n}+L_{n+1}\textbf{i}+L_{n+2}\textbf{j}+L_{n+3}\textbf{k}\ \ (L_{n}\textbf{1}=L_{n}),
\]
respectively, where $F_{n}$ and $L_{n}$ are the $n$-th classic Fibonacci and Lucas numbers, respectively, and the author studied the properties of these quaternions. Several interesting and useful extensions of many of the familiar quaternion
numbers (such as the Fibonacci and Lucas quaternions \cite{Aky,Hal1,Hor1}, Pell quaternion \cite{Ca,Cim1} and Jacobsthal quaternions \cite{Szy-Wl} have been considered by several authors.

After the work of Hamilton, James Cockle introduced the set of split quaternions which can be represented as
$$\Bbb{H}_{(1,-1)}=\{q=q_{r}+q_{i}e_{1}+q_{j}e_{2}+q_{k}e_{3}:\ q_{r},q_{s}\in \Bbb{R},\ s=i,j,k\}$$
where $e_{1}^{2}=-\textbf{1}$, $e_{2}^{2}=e_{3}^{2}=\textbf{1}$ and $e_{1}e_{2}e_{3}=\textbf{1}$. Note that $e_{1}e_{2}=e_{3}=-e_{2}e_{1}$,  $e_{2}e_{3}=- e_{1}=-e_{3}e_{2}$ and  $e_{3}e_{1}=e_{2}=-e_{1}e_{3}$. The set of split quaternions is also noncommutative. Unlike quaternion algebra, the set of split quaternions contains zero divisors, nilpotent and nontrivial idempotent elements, \cite{Ku}. For more properties of the split quaternions the reader is refereed to \cite{Oz}.

The set of generalized quaternions which can be represented as
$$\Bbb{H}_{(\alpha,\beta)}=\{q=q_{r}+q_{i}e_{1}+q_{j}e_{2}+q_{k}e_{3}:\ q_{r},q_{s}\in \Bbb{R},\ s=i,j,k\},$$ where $e_{1}$, $e_{2}$ and $e_{3}$ are quaternionic units which satisfy the equalities
\begin{equation}\label{u}
\left\lbrace
\begin{aligned}
e_{1}^{2}&=-\alpha, \ e_{2}^{2}=-\beta, \ e_{3}^{2}=-\alpha\beta,\\
e_{1}e_{2}&=e_{3}=-e_{2}e_{1}, \ e_{2}e_{3}=\beta e_{1}=-e_{3}e_{2}, \ e_{3}e_{1}=\alpha e_{2}=-e_{1}e_{3},
\end{aligned}
\right.
\end{equation}
where $\alpha, \beta \in \mathbb{R}$. 

By choosing $\alpha$ and $\beta$, there are following special cases:
\begin{itemize}[noitemsep]
\item $\alpha=\beta=1$ is considered, then $\Bbb{H}_{(1,1)}$ is the algebra of real quaternions.
\item $\alpha=1,\ \beta=-1$ is considered, then $\Bbb{H}_{(1,-1)}$ is the algebra of split quaternions.
\item $\alpha=1,\ \beta=0$ is considered, then $\Bbb{H}_{(1,0)}$ is the algebra of semi-quaternions.
\item $\alpha=-1,\ \beta=0$ is considered, then $\Bbb{H}_{(-1,0)}$ is the algebra of split semi-quaternions.
\item $\alpha=\beta=0$ is considered, then $\Bbb{H}_{(0,0)}$ is the algebra of $\frac{1}{4}$-quaternions.
\end{itemize}

Pottman and Wallner provided a brief introduction of the generalized quaternions in \cite{Po-Wa}. Furthermore, in \cite{Ja-Ya}, Jafari and Yayl\i\ studied some algebraic properties of generalized quaternions and operations over them. A generalized quaternion $q$ is a sum of a scalar and a vector, called scalar part, $S_{q}=q_{r}$, and vector part $V_{q}=q_{i}e_{1}+q_{j}e_{2}+q_{k}e_{3}\in \mathbb{R}^{3}_{\alpha\beta}$. Therefore, $\Bbb{H}_{(\alpha,\beta)}$ forms a 4-dimensional real space which contains the real axis $\mathbb{R}$ and a 3-dimensional real linear space $E^{3}_{\alpha\beta}$, so that, $\Bbb{H}_{(\alpha,\beta)}=\mathbb{R}\bigoplus E^{3}_{\alpha\beta}$ (for more details, see \cite{Ja-Ya}).

\section{Third-order Jacobsthal Quaternions}
\setcounter{equation}{0}

The Jacobsthal numbers have many interesting properties and applications in many fields of science (see, e.g., \cite{Ba,Hor3}). The Jacobsthal numbers $J_{n}$ are defined by the recurrence relation
\begin{equation}\label{e1}
J_{0}=0,\ J_{1}=1,\ J_{n+1}=J_{n}+2J_{n-1},\ n\geq1.
\end{equation}
Another important sequence is the Jacobsthal-Lucas sequence. This sequence is defined by the recurrence relation
\begin{equation}\label{ec1}
j_{0}=2,\ j_{1}=1,\ j_{n+1}=j_{n}+2j_{n-1},\ n\geq1.
\end{equation}
(see, \cite{Hor3}).

In \cite{Cook-Bac} the Jacobsthal recurrence relation is extended to higher order recurrence relations and the basic list of identities provided by A. F. Horadam \cite{Hor3} is expanded and extended to several identities for some of the higher order cases. For example, the third order Jacobsthal numbers, $\{J_{n}^{(3)}\}_{n\geq0}$, and third order Jacobsthal-Lucas numbers, $\{j_{n}^{(3)}\}_{n\geq0}$, are defined by
\begin{equation}\label{e2}
J_{n+3}^{(3)}=J_{n+2}^{(3)}+J_{n+1}^{(3)}+2J_{n}^{(3)},\ J_{0}^{(3)}=0,\ J_{1}^{(3)}=J_{2}^{(3)}=1,\ n\geq0,
\end{equation}
and 
\begin{equation}\label{e3}
j_{n+3}^{(3)}=j_{n+2}^{(3)}+j_{n+1}^{(3)}+2j_{n}^{(3)},\ j_{0}^{(3)}=2,\ j_{1}^{(3)}=1,\ j_{2}^{(3)}=5,\ n\geq0,
\end{equation}
respectively.

The following properties given for third order Jacobsthal numbers and third order Jacobsthal-Lucas numbers play important roles in this paper (see \cite{Cer,Cook-Bac}). 
\begin{equation}\label{e4}
3J_{n}^{(3)}+j_{n}^{(3)}=2^{n+1},
\end{equation}
\begin{equation}\label{e5}
j_{n}^{(3)}-3J_{n}^{(3)}=2j_{n-3}^{(3)},
\end{equation}
\begin{equation}\label{ec5}
J_{n+2}^{(3)}-4J_{n}^{(3)}=\left\{ 
\begin{array}{ccc}
-2 & \textrm{if} & \mymod{n}{1}{3} \\ 
1 & \textrm{if} & \mynotmod{n}{1}{3}%
\end{array}%
\right. ,
\end{equation}
\begin{equation}\label{e6}
j_{n}^{(3)}-4J_{n}^{(3)}=\left\{ 
\begin{array}{ccc}
2 & \textrm{if} & \mymod{n}{0}{3} \\ 
-3 & \textrm{if} & \mymod{n}{1}{3}\\ 
1 & \textrm{if} & \mymod{n}{2}{3}%
\end{array}%
\right. ,
\end{equation}
\begin{equation}\label{e7}
j_{n+1}^{(3)}+j_{n}^{(3)}=3J_{n+2}^{(3)},
\end{equation}
\begin{equation}\label{e8}
j_{n}^{(3)}-J_{n+2}^{(3)}=\left\{ 
\begin{array}{ccc}
1 & \textrm{if} & \mymod{n}{0}{3} \\ 
-1 & \textrm{if} & \mymod{n}{1}{3} \\ 
0 & \textrm{if} & \mymod{n}{2}{3}%
\end{array}%
\right. ,
\end{equation}
\begin{equation}\label{e9}
\left( j_{n-3}^{(3)}\right) ^{2}+3J_{n}^{(3)}j_{n}^{(3)}=4^{n},
\end{equation}
\begin{equation}\label{e10}
\sum\limits_{k=0}^{n}J_{k}^{(3)}=\left\{ 
\begin{array}{ccc}
J_{n+1}^{(3)} & \textrm{if} & \mynotmod{n}{0}{3} \\ 
J_{n+1}^{(3)}-1 & \textrm{if} & \mymod{n}{0}{3}%
\end{array}%
\right. ,
\end{equation}
\begin{equation}\label{e11}
\sum\limits_{k=0}^{n}j_{k}^{(3)}=\left\{ 
\begin{array}{ccc}
j_{n+1}^{(3)}-2 & \textrm{if} & \mynotmod{n}{0}{3} \\ 
j_{n+1}^{(3)}+1 & \textrm{if} & \mymod{n}{0}{3}%
\end{array}%
\right. 
\end{equation}
and
\begin{equation}\label{e12}
\left( j_{n}^{(3)}\right) ^{2}-9\left( J_{n}^{(3)}\right)^{2}=2^{n+2}j_{n-3}^{(3)}.
\end{equation}

Using standard techniques for solving recurrence relations, the auxiliary equation, and its roots are given by 
$$x^{3}-x^{2}-x-2=0;\ x = 2,\ \textrm{and}\ x=\frac{-1\pm i\sqrt{3}}{2}.$$ 

Note that the latter two are the complex conjugate cube roots of unity. Call them $\omega_{1}$ and $\omega_{2}$, respectively. Thus the Binet formulas can be written as
\begin{equation}\label{b1}
J_{n}^{(3)}=\frac{2}{7}2^{n}-\frac{3+2i\sqrt{3}}{21}\omega_{1}^{n}-\frac{3-2i\sqrt{3}}{21}\omega_{2}^{n}
\end{equation}
and
\begin{equation}\label{b2}
j_{n}^{(3)}=\frac{8}{7}2^{n}+\frac{3+2i\sqrt{3}}{7}\omega_{1}^{n}+\frac{3-2i\sqrt{3}}{7}\omega_{2}^{n},
\end{equation}
respectively. Now, we use the notation
\begin{equation}\label{h1}
V_{n}^{(3)}=\frac{A\omega_{1}^{n}-B\omega_{2}^{n}}{\omega_{1}-\omega_{2}}=\left\{ 
\begin{array}{ccc}
2 & \textrm{if} & \mymod{n}{0}{3} \\ 
-3 & \textrm{if} & \mymod{n}{1}{3} \\ 
1& \textrm{if} & \mymod{n}{2}{3}
\end{array}%
\right. ,
\end{equation}
where $A=-3-2\omega_{2}$ and $B=-3-2\omega_{1}$. Furthermore, note that for all $n\geq0$ we have 
\begin{equation}
V_{n+2}^{(3)}=-V_{n+1}^{(3)}-V_{n}^{(3)},\ V_{0}^{(3)}=2\ \textrm{and}\ V_{1}^{(3)}=-3.
\end{equation}

From the Binet formulas (\ref{b1}), (\ref{b2}) and Eq. (\ref{h1}), we have
\begin{equation}\label{h2}
J_{n}^{(3)}=\frac{1}{7}\left(2^{n+1}-V_{n}^{(3)}\right)\ \textrm{and}\ j_{n}^{(3)}=\frac{1}{7}\left(2^{n+3}+3V_{n}^{(3)}\right).
\end{equation}

In \cite{Cer}, the author introduced the so-called third order Jacobsthal quaternions, which are a new class of quaternion sequences. They are defined by
\begin{equation}\label{eq:1}
JQ_{n}^{(3)}=\sum_{s=0}^{3}J_{n+s}^{(3)}e_{s}=J_{n}^{(3)}+\sum_{s=1}^{3}J_{n+s}^{(3)}e_{s},\ (J_{n}^{(3)}\textbf{1}=J_{n}^{(3)})
\end{equation}
where $J_{n}^{(3)}$ is the $n$-th third order Jacobsthal number, $e_{1}^{2}=e_{2}^{2}=e_{3}^{2}=-\textbf{1}$ and $e_{1}e_{2}e_{3}=-\textbf{1}$.

The main objective of this paper is to define third-order Jacobsthal generalized quaternions and obtain the relations related to these quaternions. (i.e., for split third-order Jacobsthal quaternions, third-order Jacobsthal semi-quaternions and split third-order Jacobsthal semi-quaternions).

\section{Third-order Jacobsthal Generalized Quaternions}
\setcounter{equation}{0}

The third-order Jacobsthal and third-order Jacobsthal-Lucas generalized quaternions have respectively the expressions of following forms
\begin{equation}\label{f1}
JQ_{\alpha,\beta,n}^{(3)}=\sum_{s=0}^{3}J_{n+s}^{(3)}e_{s}=J_{n}^{(3)}+\sum_{s=1}^{3}J_{n+s}^{(3)}e_{s},\ (J_{n}^{(3)}\textbf{1}=J_{n}^{(3)})
\end{equation}
and
\begin{equation}\label{f2}
jQ_{\alpha,\beta,n}^{(3)}=\sum_{s=0}^{3}j_{n+s}^{(3)}e_{s}=j_{n}^{(3)}+\sum_{s=1}^{3}j_{n+s}^{(3)}e_{s},\ (j_{n}^{(3)}\textbf{1}=j_{n}^{(3)}),
\end{equation}
where $J_{n}^{(3)}$ is the $n$-th third-order Jacobsthal number, $j_{n}^{(3)}$ is the $n$-th third-order Jacobsthal-Lucas number and
$e_{1}$, $e_{2}$ and $e_{3}$ are quaternionic units which satisfy the equalities
$$
\left\lbrace
\begin{aligned}
e_{1}^{2}&=-\alpha, \ e_{2}^{2}=-\beta, \ e_{3}^{2}=-\alpha\beta,\\
e_{1}e_{2}&=e_{3}=-e_{2}e_{1}, \ e_{2}e_{3}=\beta e_{1}=-e_{3}e_{2}, \ e_{3}e_{1}=\alpha e_{2}=-e_{1}e_{3}.
\end{aligned}
\right.
$$
Let us denote the sets of the third-order Jacobsthal and third-order Jacobsthal-Lucas generalized quaternions by $JQ_{\alpha,\beta}^{(3)}$ and $jQ_{\alpha,\beta}^{(3)}$ respectively and their natural basis by choosing $\alpha$ and $\beta$:
\begin{itemize}[noitemsep]
\item For $\alpha=\beta=1$, $JQ_{1,1}^{(3)}$ is the set of third-order Jacobsthal real quaternions and $jQ_{1,1}^{(3)}$ is the set of third-order Jacobsthal-Lucas real quaternions, \cite{Cer}.
\item For $\alpha=1,\ \beta=-1$, $JQ_{1,-1}^{(3)}$ is the set of split third-order Jacobsthal quaternions and $jQ_{1,-1}^{(3)}$ is the set of split third-order Jacobsthal-Lucas quaternions.
\item For $\alpha=1,\ \beta=0$, $JQ_{1,0}^{(3)}$ is the set of third-order Jacobsthal semi-quaternions.
\item For $\alpha=-1,\ \beta=0$, $JQ_{-1,0}^{(3)}$ is the set of split third-order Jacobsthal semi-quaternions.
\item For $\alpha=\beta=0$, $JQ_{0,0}^{(3)}$ is the set of third-order Jacobsthal $\frac{1}{4}$-quaternions.
\end{itemize}
Throughout this paper, we study on third-order Jacobsthal generalized quaternions $JQ_{\alpha,\beta}^{(3)}$ . Similar relations hold for third-order Jacobsthal-Lucas generalized quaternions $jQ_{\alpha,\beta}^{(3)}$. In the following we will study the important properties of the third-order Jacobsthal generalized quaternions and third-order Jacobsthal-Lucas generalized quaternions:
\begin{itemize}[noitemsep]
\item The sum and subtract of $JQ_{\alpha,\beta,n}^{(3)}$ and $jQ_{\alpha,\beta,n}^{(3)}$ is defined as 
\begin{equation}\label{s1}
JQ_{\alpha,\beta,n}^{(3)}\pm jQ_{\alpha,\beta,n}^{(3)}=\sum_{s=0}^{7}(J_{n+s}^{(3)}\pm j_{n+s}^{(3)})e_{s},
\end{equation}
where $JQ_{\alpha,\beta,n}^{(3)},jQ_{\alpha,\beta,n}^{(3)}\in \Bbb{H}_{(\alpha,\beta)}$. \\
Furthermore, we can be written as $JQ_{\alpha,\beta,n}^{(3)}=S_{JQ_{\alpha,\beta,n}^{(3)}}+V_{JQ_{\alpha,\beta,n}^{(3)}}$, where $S_{JQ_{\alpha,\beta,n}^{(3)}}=J_{n}^{(3)}$ and $V_{JQ_{\alpha,\beta,n}^{(3)}}=\sum_{s=1}^{7}J_{n+s}^{(3)}e_{s}$ are called the scalar and vector parts, respectively.
\item The multiplication of these quaternions are defined by
\begin{equation}\label{m1}
\begin{aligned}
JQ_{\alpha,\beta,n}^{(3)}\cdot jQ_{\alpha,\beta,n}^{(3)}&=J_{n}^{(3)}j_{n}^{(3)}-h\left(V_{JQ_{\alpha,\beta,n}^{(3)}},V_{jQ_{\alpha,\beta,n}^{(3)}}\right)\\
&\ \ +J_{n}^{(3)}V_{jQ_{\alpha,\beta,n}^{(3)}}+j_{n}^{(3)}V_{JQ_{\alpha,\beta,n}^{(3)}}+V_{JQ_{\alpha,\beta,n}^{(3)}}\times V_{jQ_{\alpha,\beta,n}^{(3)}},
\end{aligned}
\end{equation}
where 
\begin{align*}
h\left(V_{JQ_{\alpha,\beta,n}^{(3)}},V_{jQ_{\alpha,\beta,n}^{(3)}}\right)&=\alpha JQ_{\alpha,\beta,n+1}^{(3)}\cdot jQ_{\alpha,\beta,n+1}^{(3)}+\beta JQ_{\alpha,\beta,n+2}^{(3)}\cdot jQ_{\alpha,\beta,n+2}^{(3)}\\
&\ \ +\alpha\beta JQ_{\alpha,\beta,n+3}^{(3)}\cdot jQ_{\alpha,\beta,n+3}^{(3)}\end{align*}
and
\begin{align*}
V_{JQ_{\alpha,\beta,n}^{(3)}}\times V_{jQ_{\alpha,\beta,n}^{(3)}}&=\beta i\left(JQ_{\alpha,\beta,2}^{(3)}\cdot jQ_{\alpha,\beta,3}^{(3)}-JQ_{\alpha,\beta,3}^{(3)}\cdot jQ_{\alpha,\beta,2}^{(3)}\right)\\
&\ \ +\alpha j\left(JQ_{\alpha,\beta,3}^{(3)}\cdot jQ_{\alpha,\beta,1}^{(3)}-JQ_{\alpha,\beta,1}^{(3)}\cdot jQ_{\alpha,\beta,3}^{(3)}\right)\\
&\ \ + k\left(JQ_{\alpha,\beta,1}^{(3)}\cdot jQ_{\alpha,\beta,2}^{(3)}-JQ_{\alpha,\beta,2}^{(3)}\cdot jQ_{\alpha,\beta,1}^{(3)}\right).
\end{align*}
\item The conjugate of $JQ_{\alpha,\beta,n}^{(3)}$ is defined by 
\begin{equation}\label{s2}
\overline{JQ_{\alpha,\beta,n}^{(3)}}=S_{JQ_{\alpha,\beta,n}^{(3)}}-V_{JQ_{\alpha,\beta,n}^{(3)}}=J_{n}^{(3)}-\sum_{s=1}^{7}J_{n+s}^{(3)}e_{s}
\end{equation}
and this operation satisfies
\begin{equation}
\begin{aligned}
\overline{\overline{JQ_{\alpha,\beta,n}^{(3)}}}&=JQ_{\alpha,\beta,n}^{(3)},\\
\overline{JQ_{\alpha,\beta,n}^{(3)}+ jQ_{\alpha,\beta,n}^{(3)}}&=\overline{JQ_{\alpha,\beta,n}^{(3)}}+\overline{jQ_{\alpha,\beta,n}^{(3)}},\\
\overline{JQ_{\alpha,\beta,n}^{(3)} \cdot jQ_{\alpha,\beta,n}^{(3)}}&=\overline{jQ_{\alpha,\beta,n}^{(3)}}\cdot \overline{JQ_{\alpha,\beta,n}^{(3)}},
\end{aligned}
\end{equation}
for all $JQ_{\alpha,\beta,n}^{(3)}, jQ_{\alpha,\beta,n}^{(3)}\in \Bbb{H}_{(\alpha,\beta)}$.
\item The norm of an third-order Jacobsthal generalized quaternion, which agrees with the standard Euclidean norm on $\Bbb{R}^{4}$ is defined as 
\begin{equation}\label{s3}
Nr(JQ_{\alpha,\beta,n}^{(3)})=\Big\lvert\overline{JQ_{\alpha,\beta,n}^{(3)}}\cdot JQ_{\alpha,\beta,n}^{(3)}\Big\rvert=\Big\lvert JQ_{\alpha,\beta,n}^{(3)}\cdot \overline{JQ_{\alpha,\beta,n}^{(3)}}\Big\rvert.
\end{equation}
\item The inverse of $JQ_{\alpha,\beta,n}^{(3)}\neq 0$ is given by $\left(JQ_{\alpha,\beta,n}^{(3)}\right)^{-1}=\frac{\overline{JQ_{\alpha,\beta,n}^{(3)}}}{Nr(JQ_{\alpha,\beta,n}^{(3)})}$. From the above two definitions it is deduced that 
\begin{equation}\label{s4}
\left(JQ_{\alpha,\beta,n}^{(3)}\cdot jQ_{\alpha,\beta,n}^{(3)}\right)^{-1}=\left(jQ_{\alpha,\beta,n}^{(3)}\right)^{-1}\cdot \left(JQ_{\alpha,\beta,n}^{(3)}\right)^{-1}.
\end{equation}
\end{itemize}

Now, by the addition, subtraction and multiplication we can give the following theorems.
\begin{theorem}\label{th:1}
Let $JQ_{\alpha,\beta,n}^{(3)}$ be the third-order Jacobsthal generalized quaternion. For $n\geq 1$, the following relations hold
\begin{equation}\label{eq:1}
2JQ_{\alpha,\beta,n}^{(3)}+JQ_{\alpha,\beta,n+1}^{(3)}+JQ_{\alpha,\beta,n+2}^{(3)}=JQ_{\alpha,\beta,n+3}^{(3)},
\end{equation}
\begin{equation}\label{eq:2}
\begin{aligned}
JQ_{\alpha,\beta,n}^{(3)}-e_{1}JQ_{\alpha,\beta,n+1}^{(3)}&-e_{2}JQ_{\alpha,\beta,n+2}^{(3)}-e_{3}JQ_{\alpha,\beta,n+3}^{(3)}\\
&=\left\lbrace
\begin{array}{c}
(1+2\beta +10\alpha\beta)J_{n}^{(3)}+(3\beta+9\alpha\beta)J_{n+1}^{(3)}\\
+(\alpha+2\beta+9\alpha\beta)J_{n+2}^{(3)}
\end{array}
\right\rbrace,
\end{aligned}
\end{equation}
where $\alpha$ and $\beta$ are real numbers and $J_{n}^{(3)}$ is the $n$-th third-order Jacobsthal number.
\end{theorem}
\begin{proof}
(\ref{eq:1}): By the using the equations (\ref{f1}) and (\ref{s1}), we have
\begin{align*}
&2JQ_{\alpha,\beta,n}^{(3)}+JQ_{\alpha,\beta,n+1}^{(3)}+JQ_{\alpha,\beta,n+2}^{(3)}\\
&=\left(2J_{n}^{(3)}+\sum_{s=1}^{3}2J_{n+s}^{(3)}e_{s}\right)+\left(J_{n+1}^{(3)}+\sum_{s=1}^{3}J_{n+s+1}^{(3)}e_{s}\right)+ \left(J_{n+2}^{(3)}+\sum_{s=1}^{3}J_{n+s+2}^{(3)}e_{s}\right)\\
&=\left(2J_{n}^{(3)}+J_{n+1}^{(3)}+J_{n+2}^{(3)}\right)+\sum_{s=1}^{3}\left(2J_{n+s}^{(3)}+J_{n+s+1}^{(3)}+J_{n+s+2}^{(3)}\right)e_{s}\\
&=J_{n+3}^{(3)}+\sum_{s=1}^{3}J_{n+s+3}^{(3)}e_{s}=JQ_{\alpha,\beta,n+3}^{(3)}.
\end{align*}
Using the identity of third-order Jacobsthal numbers $$J_{n+3}^{(3)}=J_{n+2}^{(3)}+J_{n+1}^{(3)}+2J_{n}^{(3)},\ (n\geq 0)$$ in (\ref{e2}), the last equation becomes $2JQ_{\alpha,\beta,n}^{(3)}+JQ_{\alpha,\beta,n+1}^{(3)}+JQ_{\alpha,\beta,n+2}^{(3)}=JQ_{\alpha,\beta,n+3}^{(3)}$.

(\ref{eq:2}): From equations (\ref{f1}) and (\ref{u}), we conclude that
\begin{align*}
JQ_{\alpha,\beta,n}^{(3)}&-e_{1}JQ_{\alpha,\beta,n+1}^{(3)}-e_{2}JQ_{\alpha,\beta,n+2}^{(3)}-e_{3}JQ_{\alpha,\beta,n+3}^{(3)}\\
&=\left(J_{n}^{(3)}+\sum_{s=1}^{3}J_{n+s}^{(3)}e_{s}\right)-e_{1}\left(J_{n+1}^{(3)}+\sum_{s=1}^{3}J_{n+s+1}^{(3)}e_{s}\right)\\
&\ \ -e_{2}\left(J_{n+2}^{(3)}+\sum_{s=1}^{3}J_{n+s+2}^{(3)}e_{s}\right)-e_{3}\left(J_{n+3}^{(3)}+\sum_{s=1}^{3}J_{n+s+3}^{(3)}e_{s}\right)\\
&=J_{n}^{(3)}+\alpha J_{n+2}^{(3)}+\beta J_{n+4}^{(3)}+\alpha\beta J_{n+6}^{(3)}.
\end{align*}
Substituting the identities of third-order Jacobsthal numbers $J_{n+4}^{(3)}=2J_{n+2}^{(3)}+3J_{n+1}^{(3)}+2J_{n}^{(3)}$ and $J_{n+6}^{(3)}=9J_{n+2}^{(3)}+10J_{n+1}^{(3)}+9J_{n}^{(3)}$ which are well-known using relation (\ref{s1}) into the last equation and after simplifying we can assert that
\begin{align*}
JQ_{\alpha,\beta,n}^{(3)}&-e_{1}JQ_{\alpha,\beta,n+1}^{(3)}-e_{2}JQ_{\alpha,\beta,n+2}^{(3)}-e_{3}JQ_{\alpha,\beta,n+3}^{(3)}\\
&=\left\lbrace
\begin{array}{c}
(1+2\beta +10\alpha\beta)J_{n}^{(3)}+(3\beta+9\alpha\beta)J_{n+1}^{(3)}\\
+(\alpha+2\beta+9\alpha\beta)J_{n+2}^{(3)}
\end{array}
\right\rbrace.
\end{align*}
Special Cases:
\begin{itemize}[noitemsep]
\item For $\alpha=\beta=1$, the equation (\ref{eq:2}) is equivalent to
$$\left\lbrace
\begin{array}{c}
JQ_{1,1,n}^{(3)}-e_{1}JQ_{1,1,n+1}^{(3)}\\
-e_{2}JQ_{1,1,n+2}^{(3)}-e_{3}JQ_{1,1,n+3}^{(3)}\end{array}
\right\rbrace=37J_{n}^{(3)}+12j_{n}^{(3)},$$
wich was given by Cerda-Morales in \cite{Cer}.
\item For the case $\alpha=1$ and $\beta=-1$, the equation (\ref{eq:2}) becomes
$$\left\lbrace
\begin{array}{c}
JQ_{1,-1,n}^{(3)}-e_{1}JQ_{1,-1,n+1}^{(3)}\\
-e_{2}JQ_{1,-1,n+2}^{(3)}-e_{3}JQ_{1,-1,n+3}^{(3)}\end{array}
\right\rbrace=-\left(J_{n}^{(3)}+17J_{n+1}^{(3)}+5j_{n+1}^{(3)}\right).$$
\item Let $\beta=0$. For $\alpha=1$, $\alpha=-1$ and $\alpha=0$, there are following relations
$$JQ_{1,0,n}^{(3)}-e_{1}JQ_{1,0,n+1}^{(3)}-e_{2}JQ_{1,0,n+2}^{(3)}-e_{3}JQ_{1,0,n+3}^{(3)}=J_{n}^{(3)}+J_{n+2}^{(3)},$$
$$JQ_{-1,0,n}^{(3)}-e_{1}JQ_{-1,0,n+1}^{(3)}-e_{2}JQ_{-1,0,n+2}^{(3)}-e_{3}JQ_{-1,0,n+3}^{(3)}=J_{n}^{(3)}-J_{n+2}^{(3)}$$ and
$$JQ_{0,0,n}^{(3)}-e_{1}JQ_{0,0,n+1}^{(3)}-e_{2}JQ_{0,0,n+2}^{(3)}-e_{3}JQ_{0,0,n+3}^{(3)}=J_{n}^{(3)},$$
respectively.
\end{itemize}
\end{proof}

\begin{theorem}\label{th:2}
Let $J_{n}^{(3)}$ and $j_{n}^{(3)}$ be the third-order Jacobsthal and third-order Jacobsthal-Lucas numbers, $JQ_{\alpha,\beta,n}^{(3)}$ and $jQ_{\alpha,\beta,n}^{(3)}$ be the third-order Jacobsthal and third-order Jacobsthal-Lucas generalized quaternions, respectively. In this case, the following equations can be given
\begin{equation}\label{eq:3}
\begin{aligned}
&\left(JQ_{\alpha,\beta,n}^{(3)}\right)^{2}+\left(JQ_{\alpha,\beta,n+1}^{(3)}\right)^{2}+\left(JQ_{\alpha,\beta,n+2}^{(3)}\right)^{2}\\
&=\frac{1}{7}\left\lbrace
\begin{array}{c}
14\left(J_{n}^{(3)}\cdot JQ_{\alpha,\beta,n}^{(3)}+J_{n+1}^{(3)}\cdot JQ_{\alpha,\beta,n+1}^{(3)}+J_{n+2}^{(3)}\cdot JQ_{\alpha,\beta,n+2}^{(3)}\right)\\
-3\cdot 2^{2(n+1)}(1+4\alpha+16\beta +64\alpha\beta)\\
-\cdot 2^{n+2}\left(U_{n}^{(3)}+2\alpha U_{n+1}^{(3)}+4\beta U_{n+2}^{(3)}+8\alpha\beta U_{n+3}^{(3)}\right)\\
-2(1+\alpha+\beta +\alpha\beta)
\end{array}
\right\rbrace.
\end{aligned}
\end{equation}
\begin{equation}\label{eq:4}
\begin{aligned}
\left(jQ_{\alpha,\beta,n}^{(3)}\right)^{2}&-9\left(JQ_{\alpha,\beta,n}^{(3)}\right)^{2}\\
&=\left\lbrace
\begin{array}{c}
2j_{n}^{(3)}\cdot jQ_{\alpha,\beta,n}^{(3)}-18J_{n}^{(3)}\cdot JQ_{\alpha,\beta,n}^{(3)}\\
-2^{n+2}\left(j_{n-3}^{(3)}+2\alpha j_{n-2}^{(3)} +4\beta j_{n-1}^{(3)} +8\alpha\beta j_{n}^{(3)}\right)
\end{array}
\right\rbrace,
\end{aligned}
\end{equation}
where $\alpha$ and $\beta$ are real numbers, and $U_{n}^{(3)}=j_{n-1}^{(3)}-J_{n+1}^{(3)}$.
\end{theorem}
\begin{proof}
(\ref{eq:3}): From equation (\ref{f1}), we get
\begin{align*}
\left(JQ_{\alpha,\beta,n}^{(3)}\right)^{2}&=\left(J_{n}^{(3)}+\sum_{s=1}^{3}J_{n+s}^{(3)}e_{s}\right)\cdot \left(J_{n}^{(3)}+\sum_{s=1}^{3}J_{n+s}^{(3)}e_{s}\right)\\
&=-\left(\left(J_{n}^{(3)}\right)^{2}+\alpha \left(J_{n+1}^{(3)}\right)^{2}+\beta \left(J_{n+2}^{(3)}\right)^{2}+\alpha \beta \left(J_{n+3}^{(3)}\right)^{2}\right)\\
&\ \ + 2J_{n}^{(3)}\cdot JQ_{\alpha,\beta,n}^{(3)}.
\end{align*}
Combining the equation (\ref{s1}) with the last equation gives
\begin{align*}
&\left(JQ_{\alpha,\beta,n}^{(3)}\right)^{2}+\left(JQ_{\alpha,\beta,n+1}^{(3)}\right)^{2}+\left(JQ_{\alpha,\beta,n+2}^{(3)}\right)^{2}\\
&=-\left(\left(J_{n}^{(3)}\right)^{2}+ \left(J_{n+1}^{(3)}\right)^{2}+\left(J_{n+2}^{(3)}\right)^{2}\right)\\
&\ \ - \alpha \left(\left(J_{n+1}^{(3)}\right)^{2}+ \left(J_{n+2}^{(3)}\right)^{2}+\left(J_{n+3}^{(3)}\right)^{2}\right)\\
&\ \ - \beta \left(\left(J_{n+2}^{(3)}\right)^{2}+ \left(J_{n+3}^{(3)}\right)^{2}+\left(J_{n+4}^{(3)}\right)^{2}\right)\\
&\ \ - \alpha \beta \left(\left(J_{n+3}^{(3)}\right)^{2}+ \left(J_{n+4}^{(3)}\right)^{2}+\left(J_{n+5}^{(3)}\right)^{2}\right)\\
&\ \ + 2\left(J_{n}^{(3)}\cdot JQ_{\alpha,\beta,n}^{(3)}+J_{n+1}^{(3)}\cdot JQ_{\alpha,\beta,n+1}^{(3)}+J_{n+2}^{(3)}\cdot JQ_{\alpha,\beta,n+2}^{(3)}\right).
\end{align*}
Thus, using the identity of the third-order Jacobsthal numbers $\left(J_{n}^{(3)}\right)^{2}+ \left(J_{n+1}^{(3)}\right)^{2}+\left(J_{n+2}^{(3)}\right)^{2}=\frac{1}{7}\left(3\cdot 2^{2(n+1)}-2^{n+2}U_{n}^{(3)}+2\right)$ by (\ref{b1}) into the last equation and after simplifying we obtain that
\begin{align*}
&\left(JQ_{\alpha,\beta,n}^{(3)}\right)^{2}+\left(JQ_{\alpha,\beta,n+1}^{(3)}\right)^{2}+\left(JQ_{\alpha,\beta,n+2}^{(3)}\right)^{2}\\
&=\left\lbrace
\begin{array}{c}
2\left(J_{n}^{(3)}\cdot JQ_{\alpha,\beta,n}^{(3)}+J_{n+1}^{(3)}\cdot JQ_{\alpha,\beta,n+1}^{(3)}+J_{n+2}^{(3)}\cdot JQ_{\alpha,\beta,n+2}^{(3)}\right)\\
-\frac{3}{7}\cdot 2^{2(n+1)}(1+4\alpha+16\beta +64\alpha\beta)\\
-\frac{1}{7}\cdot 2^{n+2}\left(U_{n}^{(3)}+2\alpha U_{n+1}^{(3)}+4\beta U_{n+2}^{(3)}+8\alpha\beta U_{n+3}^{(3)}\right)\\
-\frac{2}{7}(1+\alpha+\beta +\alpha\beta)
\end{array}
\right\rbrace,
\end{align*}
where $U_{n}^{(3)}=j_{n-1}^{(3)}-J_{n+1}^{(3)}$.

(\ref{eq:4}): In the same manner, from the equations (\ref{f1}), (\ref{f2}) and (\ref{s1}) we can see that
\begin{align*}
&\left(jQ_{\alpha,\beta,n}^{(3)}\right)^{2}-9\left(JQ_{\alpha,\beta,n}^{(3)}\right)^{2}\\
&=2j_{n}^{(3)}\cdot jQ_{\alpha,\beta,n}^{(3)}-18J_{n}^{(3)}\cdot JQ_{\alpha,\beta,n}^{(3)}\\
&\ \ -\left(\left(j_{n}^{(3)}\right)^{2}-9\left(J_{n}^{(3)}\right)^{2}\right)-\alpha \left(\left(j_{n+1}^{(3)}\right)^{2}-9\left(J_{n+1}^{(3)}\right)^{2}\right)\\
&\ \ -\beta \left(\left(j_{n+2}^{(3)}\right)^{2}-9\left(J_{n+2}^{(3)}\right)^{2}\right)-\alpha \beta \left(\left(j_{n+3}^{(3)}\right)^{2}-9\left(J_{n+3}^{(3)}\right)^{2}\right).
\end{align*}
Putting the identity $\left( j_{n}^{(3)}\right) ^{2}-9\left( J_{n}^{(3)}\right)^{2}=2^{n+2}j_{n-3}^{(3)}$ of the third-order Jacobsthal numbers (see \cite{Cook-Bac}) into the last equations and after an easy computation we obtain
\begin{align*}
\left(jQ_{\alpha,\beta,n}^{(3)}\right)^{2}-9\left(JQ_{\alpha,\beta,n}^{(3)}\right)^{2}&=2j_{n}^{(3)}\cdot jQ_{\alpha,\beta,n}^{(3)}-18J_{n}^{(3)}\cdot JQ_{\alpha,\beta,n}^{(3)}\\
&\ \ -2^{n+2}\left(j_{n-3}^{(3)}+2\alpha j_{n-2}^{(3)} +4\beta j_{n-1}^{(3)} +8\alpha\beta j_{n}^{(3)}\right).
\end{align*}

Special cases:
\begin{itemize}[noitemsep]
\item For $\alpha=\beta=1$, we have the following relations for the third-order Jacobsthal quaternions which were given by Cerda-Morales in \cite{Cer}.
$$\left(jQ_{1,1,n}^{(3)}\right)^{2}-9\left(JQ_{1,1,n}^{(3)}\right)^{2}=2\left\lbrace
\begin{array}{c}
j_{n}^{(3)}\cdot jQ_{1,1,n}^{(3)}-9J_{n}^{(3)}\cdot JQ_{1,1,n}^{(3)}\\
-2^{n}\left(17j_{n}^{(3)}+7j_{n-1}^{(3)} +3 j_{n-2}^{(3)}\right)
\end{array}
\right\rbrace.$$
\item Let $\alpha=1$ and $\beta=-1$. In this case, we have the following relations for the split third-order Jacobsthal quaternions: 
$$\left(jQ_{1,-1,n}^{(3)}\right)^{2}-9\left(JQ_{1,-1,n}^{(3)}\right)^{2}=2\left\lbrace
\begin{array}{c}
j_{n}^{(3)}\cdot jQ_{1,-1,n}^{(3)}-9J_{n}^{(3)}\cdot JQ_{1,-1,n}^{(3)}\\
+3\cdot 2^{n}\left(5j_{n}^{(3)}+3j_{n-1}^{(3)}-j_{n-2}^{(3)} \right)
\end{array}
\right\rbrace.$$
\item Let $\beta=0$. For $\alpha=1$, $\alpha=-1$ and $\alpha=0$, the equation (\ref{eq:4}) becomes
$$\left(jQ_{1,0,n}^{(3)}\right)^{2}-9\left(JQ_{1,0,n}^{(3)}\right)^{2}=2\left\lbrace
\begin{array}{c}
j_{n}^{(3)}\cdot jQ_{1,0,n}^{(3)}-9J_{n}^{(3)}\cdot JQ_{1,0,n}^{(3)}\\
-2^{n+1}\left(j_{n}^{(3)}-j_{n-1}^{(3)}\right)
\end{array}
\right\rbrace,$$
$$\left(jQ_{-1,0,n}^{(3)}\right)^{2}-9\left(JQ_{-1,0,n}^{(3)}\right)^{2}=2\left\lbrace
\begin{array}{c}
j_{n}^{(3)}\cdot jQ_{-1,0,n}^{(3)}-9J_{n}^{(3)}\cdot JQ_{-1,0,n}^{(3)}\\
-2^{n+1}\left(j_{n-3}^{(3)}-2j_{n-2}^{(3)}\right)
\end{array}
\right\rbrace$$ and
$$\left(jQ_{0,0,n}^{(3)}\right)^{2}-9\left(JQ_{0,0,n}^{(3)}\right)^{2}=2\left\lbrace
\begin{array}{c}
j_{n}^{(3)}\cdot jQ_{0,0,n}^{(3)}-9J_{n}^{(3)}\cdot JQ_{0,0,n}^{(3)}-2^{n+1}j_{n-3}^{(3)}
\end{array}
\right\rbrace,$$
respectively.
\end{itemize}
\end{proof}

In Theorem \ref{th:3}, the first identity of norm for $\alpha=\beta=1$ is analogous to the ordinary third-order Jacobsthal quaternions 
\begin{equation}\label{t1}
\cdot Nr(JQ_{n}^{(3)})=\frac{1}{49}\left\{ 
\begin{array}{ccc}
340\cdot 2^{2n}-64\cdot 2^{n}+18 & \textrm{if} & \mymod{n}{0}{3} \\ 
340\cdot 2^{2n}+68\cdot 2^{n}+23& \textrm{if} & \mymod{n}{1}{3} \\ 
340\cdot 2^{2n}-4\cdot 2^{n}+15& \textrm{if} & \mymod{n}{2}{3}
\end{array}%
\right. ,
\end{equation}
(for more details, see \cite{Cer}).

\begin{theorem}\label{th:3}
Let $J_{n}^{(3)}$ be the third-order Jacobsthal number, $JQ_{\alpha,\beta,n}^{(3)}$ be the third-order Jacobsthal generalized quaternion and $\overline{JQ_{\alpha,\beta,n}^{(3)}}$ be the conjugate of $JQ_{\alpha,\beta,n}^{(3)}$. Then, the following equation hold
\begin{equation}\label{eq:5}
Nr(JQ_{\alpha,\beta,n}^{(3)})=\frac{1}{49}\left\lbrace
\begin{array}{c}
2^{2(n+1)}(1+4\alpha+16\beta+64\alpha\beta)\\
-2^{n+2}\left((1+8\alpha\beta-4\beta)V_{n}^{(3)}+(2\alpha-4\beta) V_{n+1}^{(3)}\right)\\
+(1+\alpha\beta)\left(V_{n}^{(3)}\right)^{2}+\alpha \left(V_{n+1}^{(3)}\right)^{2}+\beta \left(V_{n+2}^{(3)}\right)^{2}
\end{array}
\right\rbrace,
\end{equation}
and $V_{n}^{(3)}$ as in Eq. (\ref{h1}).
\end{theorem}
\begin{proof}
By multiplication of two third-order Jacobsthal generalized quaternions, and by using the identity of the third-order Jacobsthal numbers $J_{n+3}^{(3)}=J_{n+2}^{(3)}+J_{n+1}^{(3)}+2J_{n}^{(3)}$ it may be concluded that
\begin{align*}
Nr(JQ_{\alpha,\beta,n}^{(3)})&=\left(J_{n}^{(3)}+\sum_{s=1}^{3}J_{n+s}^{(3)}e_{s}\right)\cdot \left(J_{n}^{(3)}-\sum_{s=1}^{3}J_{n+s}^{(3)}e_{s}\right)\\
&=\left(J_{n}^{(3)}\right)^{2}+\alpha \left(J_{n+1}^{(3)}\right)^{2}+\beta \left(J_{n+2}^{(3)}\right)^{2}+\alpha \beta \left(J_{n+3}^{(3)}\right)^{2}.
\end{align*}
Finally, from the Binet formula (\ref{h2}) of $J_{n}^{(3)}$ it is obvious that
$$\left(J_{n}^{(3)}\right)^{2}=\frac{1}{49}\left(2^{n+1}-V_{n}^{(3)}\right)^{2}=\frac{1}{49}\left(2^{2(n+1)}-2^{n+2}V_{n}^{(3)}+\left(V_{n}^{(3)}\right)^{2}\right).$$
Then, we have
\begin{align*}
Nr(JQ_{\alpha,\beta,n}^{(3)})&=\frac{1}{49}\left\lbrace
\begin{array}{c}
2^{2(n+1)}(1+4\alpha+16\beta+64\alpha\beta)\\
-2^{n+2}\left(V_{n}^{(3)}+2\alpha V_{n+1}^{(3)}+4\beta V_{n+2}^{(3)}+8\alpha\beta V_{n+3}^{(3)}\right)\\
+\left(V_{n}^{(3)}\right)^{2}+\alpha \left(V_{n+1}^{(3)}\right)^{2}+\beta \left(V_{n+2}^{(3)}\right)^{2}+\alpha\beta \left(V_{n+3}^{(3)}\right)^{2}
\end{array}
\right\rbrace\\
&=\frac{1}{49}\left\lbrace
\begin{array}{c}
2^{2(n+1)}(1+4\alpha+16\beta+64\alpha\beta)\\
-2^{n+2}\left((1+8\alpha\beta-4\beta)V_{n}^{(3)}+(2\alpha-4\beta) V_{n+1}^{(3)}\right)\\
+(1+\alpha\beta)\left(V_{n}^{(3)}\right)^{2}+\alpha \left(V_{n+1}^{(3)}\right)^{2}+\beta \left(V_{n+2}^{(3)}\right)^{2}
\end{array}
\right\rbrace,
\end{align*}
using the relations $V_{n}^{(3)}+V_{n+1}^{(3)}+V_{n+2}^{(3)}=0$ and $V_{n}^{(3)}=V_{n+3}^{(3)}$ for $n\geq 0$.

Special Cases:\\
By scrutinizing $\alpha$ and $\beta$, the equation (\ref{eq:5}) becomes as follow:
\begin{itemize}[noitemsep]
\item $$Nr(JQ_{1,1,n}^{(3)})=\frac{1}{49}\left\lbrace
\begin{array}{c}
85\cdot 2^{2(n+1)}-2^{n+2}\left(5V_{n}^{(3)}-2V_{n+1}^{(3)}\right)+\left(V_{n}^{(3)}\right)^{2}+14
\end{array}
\right\rbrace,$$
\item $$Nr(JQ_{1,-1,n}^{(3)})=\frac{1}{49}\left\lbrace
\begin{array}{c}
-75\cdot 2^{2(n+1)}-3\cdot 2^{n+2}\left(2V_{n+1}^{(3)}-V_{n}^{(3)}\right)\\
+\left(V_{n+1}^{(3)}\right)^{2}-\left(V_{n+2}^{(3)}\right)^{2}
\end{array}
\right\rbrace,$$
\item $$Nr(JQ_{1,0,n}^{(3)})=\frac{1}{49}\left\lbrace
\begin{array}{c}
5\cdot 2^{2(n+1)}-2^{n+2}\left(V_{n}^{(3)}+2V_{n+1}^{(3)}\right)\\
+\left(V_{n}^{(3)}\right)^{2}+\left(V_{n+1}^{(3)}\right)^{2}
\end{array}
\right\rbrace,$$
\item $$Nr(JQ_{-1,0,n}^{(3)})=\frac{1}{49}\left\lbrace
\begin{array}{c}
-3\cdot 2^{2(n+1)}-2^{n+2}\left(V_{n}^{(3)}-2V_{n+1}^{(3)}\right)\\
+\left(V_{n}^{(3)}\right)^{2}-\left(V_{n+1}^{(3)}\right)^{2}
\end{array}
\right\rbrace,$$
\item $$Nr(JQ_{0,0,n}^{(3)})=\frac{1}{49}\left\lbrace
\begin{array}{c}
2^{2(n+1)}-2^{n+2}V_{n}^{(3)}+\left(V_{n}^{(3)}\right)^{2}
\end{array}
\right\rbrace,$$
\end{itemize}
\end{proof}

In the following theorem, the first and second formulas are analogous to the Theorem 3.3 and 3.4 in \cite{Cer}.
\begin{theorem}[Binet's Formulas]\label{th:4}
Let $\widehat{2}=1+2e_{1}+4e_{2}+8e_{3}$, $\widehat{\omega_{1}}=1+\omega_{1}e_{1}+\omega_{1}^{2}e_{2}+e_{3}$ and $\widehat{\omega_{2}}=1+\omega_{2}e_{1}+\omega_{2}^{2}e_{2}+e_{3}$ generalized quaternions. Let $JQ_{\alpha,\beta,n}^{(3)}$ and $jQ_{\alpha,\beta,n}^{(3)}$ be the third-orderJacobsthal and third-order Jacobsthal-Lucas generalized quaternions, respectively. For $n\geq0$, the Binet formulas for these quaternions are given as:
\begin{equation}
JQ_{\alpha,\beta,n}^{(3)}=\frac{1}{7}\left(2^{n+1}\widehat{2}-VQ_{n}^{(3)}\right)
\end{equation}
and
\begin{equation}
jQ_{\alpha,\beta,n}^{(3)}=\frac{1}{7}\left(2^{n+3}\widehat{2}+3VQ_{n}^{(3)}\right),
\end{equation}
respectively. Here, the sequence $VQ_{n}^{(3)}$ is defined by
\begin{equation}\label{q1}
VQ_{n}^{(3)}=\frac{A\omega_{1}^{n}\widehat{\omega_{1}}-B\omega_{2}^{n}\widehat{\omega_{2}}}{\omega_{1}-\omega_{2}}=\left\{ 
\begin{array}{ccc}
2-3e_{1}+e_{2}+2e_{3} & \textrm{if} & \mymod{n}{0}{3} \\ 
-3+e_{1}+2e_{2}-3e_{3} & \textrm{if} & \mymod{n}{1}{3} \\ 
1+2e_{1}-3e_{2}+e_{3}& \textrm{if} & \mymod{n}{2}{3}
\end{array}
\right. ,
\end{equation}
where $A=-3-2\omega_{2}$ and $B=-3-2\omega_{1}$. Furthermore, note that for all $n\geq0$ we have $VQ_{n+2}^{(3)}=-VQ_{n+1}^{(3)}-VQ_{n}^{(3)}$.
\end{theorem}

The following theorem gives d'Ocagne's identities for third-order Jacobsthal generalized quaternion.
\begin{theorem}\label{th:5}
If $JQ_{\alpha,\beta,n}^{(3)}$ be the $n$-th third-order Jacobsthal generalized quaternion. Then, for any integers $n$ and $m$, we have
\begin{equation}\label{p10}
JQ_{\alpha,\beta,m}^{(3)}\cdot JQ_{\alpha,\beta,n+1}^{(3)}-JQ_{\alpha,\beta,m+1}^{(3)}\cdot JQ_{\alpha,\beta,n}^{(3)}=\frac{1}{7}\left\lbrace\begin{array}{ccc}2^{m+1}\widehat{2} UQ_{n+1}^{(3)}-2^{n+1}UQ_{m+1}^{(3)}\widehat{2}\\
-\frac{\sqrt{3}}{3}i\left(\omega_{1}^{m-n}\widehat{\omega_{1}}\widehat{\omega_{2}}-\omega_{2}^{m-n}\widehat{\omega_{2}}\widehat{\omega_{1}}\right)\end{array}
\right\rbrace
\end{equation}
where $\widehat{2}=1+2e_{1}+4e_{2}+8e_{3}$, $\widehat{\omega_{1}}=1+\omega_{1}e_{1}+\omega_{1}^{2}e_{2}+e_{3}$, $\widehat{\omega_{2}}=1+\omega_{2}e_{1}+\omega_{2}^{2}e_{2}+e_{3}$ and $UQ_{n}^{(3)}=jQ_{\alpha,\beta,n-1}^{(3)}-JQ_{\alpha,\beta,n+1}^{(3)}$.
\end{theorem}
\begin{proof}
Using the Binet formula for the third-order Jacobsthal generalized quaternions and $VQ_{n}^{(3)}$ in (\ref{q1}) gives
\begin{equation}\label{r1}
\begin{aligned}
&JQ_{\alpha,\beta,m}^{(3)}\cdot JQ_{\alpha,\beta,n+1}^{(3)}-JQ_{\alpha,\beta,m+1}^{(3)}\cdot JQ_{\alpha,\beta,n}^{(3)}\\
&=\frac{1}{49}\left\lbrace\begin{array}{ccc}\left(2^{m+1}\widehat{2}-VQ_{m}^{(3)}\right)\left(2^{n+2}\widehat{2}-VQ_{n+1}^{(3)}\right)\\
-\left(2^{m+2}\widehat{2}-VQ_{m+1}^{(3)}\right)\left(2^{n+1}\widehat{2}-VQ_{n}^{(3)}\right)\end{array}
\right\rbrace\\
&=\frac{1}{49}\left\lbrace 
\begin{array}{ccc}
-2^{m+1}\widehat{2} VQ_{n+1}^{(3)}-2^{n+2}VQ_{m}^{(3)}\widehat{2}+2^{m+2}\widehat{2} VQ_{n}^{(3)}+2^{n+1}VQ_{m+1}^{(3)}\widehat{2}\\
+VQ_{m}^{(3)}VQ_{n+1}^{(3)}-VQ_{m+1}^{(3)}VQ_{n}^{(3)}
\end{array}
\right\rbrace\\
&=\frac{1}{7}\left(2^{m+1}\widehat{2} UQ_{n+1}^{(3)}-2^{n+1}UQ_{m+1}^{(3)}\widehat{2}-\frac{\sqrt{3}}{3}i\left(\omega_{1}^{m-n}\widehat{\omega_{1}}\widehat{\omega_{2}}-\omega_{2}^{m-n}\widehat{\omega_{2}}\widehat{\omega_{1}}\right)\right),
\end{aligned}
\end{equation}
where $UQ_{n}^{(3)}=jQ_{\alpha,\beta,n-1}^{(3)}-JQ_{\alpha,\beta,n+1}^{(3)}$ for all $n\geq 1$. 
\end{proof}

Taking $m=n+1$ in the Theorem \ref{th:5} and using the identity 
\begin{align*}
\omega_{1}\widehat{\omega_{1}}\widehat{\omega_{2}}&-\omega_{2}\widehat{\omega_{2}}\widehat{\omega_{1}}\\
&=(\omega_{1}-\omega_{2})\left\lbrace 
\begin{array}{ccc}(2-e_{1}-e_{2}+2e_{3})-(1+\alpha+\beta+\alpha\beta)\\
+(\beta e_{1}+\alpha e_{2} +e_{3})\end{array}
\right\rbrace.
\end{align*}
we obtain a type of Cassini-like identity for third-order Jacobsthal generalized quaternions.
\begin{corollary}\label{c1}
For any integer $n\geq 0$, we have
\begin{equation}\label{p11}
\left(JQ_{\alpha,\beta,n+1}^{(3)}\right)^{2}-JQ_{\alpha,\beta,n+2}^{(3)}\cdot JQ_{\alpha,\beta,n}^{(3)}=\frac{1}{7}\left\lbrace\begin{array}{ccc}2^{n+1}\left(2\widehat{2} UQ_{n+1}^{(3)}-UQ_{n+2}^{(3)}\widehat{2}\right)\\
+(2-e_{1}-e_{2}+2e_{3})\\
-(1+\alpha+\beta+\alpha\beta)\\
+(\beta e_{1}+\alpha e_{2} +e_{3})\end{array}
\right\rbrace.
\end{equation}
\end{corollary}

\textbf{Special Cases:} In particular, for the third-order Jacobsthal and third-order Jacobsthal-Lucas generalized quaternions by considering the special cases of $\alpha$ and $\beta$, respectively, the Cassini-like Identities are as follows:
\begin{itemize}[noitemsep]
\item $\left(JQ_{1,1,n+1}^{(3)}\right)^{2}-JQ_{1,1,n+2}^{(3)}\cdot JQ_{1,1,n}^{(3)}=\frac{1}{7}\left\lbrace\begin{array}{ccc}2^{n+1}\left(2\widehat{2} UQ_{n+1}^{(3)}-UQ_{n+2}^{(3)}\widehat{2}\right)\\
-2+3e_{3}\end{array}
\right\rbrace,$
\item $\left(JQ_{1,-1,n+1}^{(3)}\right)^{2}-JQ_{1,-1,n+2}^{(3)}\cdot JQ_{1,-1,n}^{(3)}=\frac{1}{7}\left\lbrace\begin{array}{ccc}2^{n+1}\left(2\widehat{2} UQ_{n+1}^{(3)}-UQ_{n+2}^{(3)}\widehat{2}\right)\\
+2-2e_{2}+3e_{3}\end{array}
\right\rbrace,$
\item $\left(JQ_{1,0,n+1}^{(3)}\right)^{2}-JQ_{1,0,n+2}^{(3)}\cdot JQ_{1,0,n}^{(3)}=\frac{1}{7}\left\lbrace\begin{array}{ccc}2^{n+1}\left(2\widehat{2} UQ_{n+1}^{(3)}-UQ_{n+2}^{(3)}\widehat{2}\right)\\
-e_{1}+3e_{3}\end{array}
\right\rbrace,$
\item $\left(JQ_{-1,0,n+1}^{(3)}\right)^{2}-JQ_{-1,0,n+2}^{(3)}\cdot JQ_{-1,0,n}^{(3)}=\frac{1}{7}\left\lbrace\begin{array}{ccc}2^{n+1}\left(2\widehat{2} UQ_{n+1}^{(3)}-UQ_{n+2}^{(3)}\widehat{2}\right)\\
2-e_{1}-2e_{2}+3e_{3}\end{array}
\right\rbrace,$
\item $\left(JQ_{0,0,n+1}^{(3)}\right)^{2}-JQ_{0,0,n+2}^{(3)}\cdot JQ_{0,0,n}^{(3)}=\frac{1}{7}\left\lbrace\begin{array}{ccc}2^{n+1}\left(2\widehat{2} UQ_{n+1}^{(3)}-UQ_{n+2}^{(3)}\widehat{2}\right)\\
-1+3e_{3}\end{array}
\right\rbrace.$
\end{itemize}

\section{Conclusions}\label{sec:4}
\setcounter{equation}{0}

The third-order Jacobsthal generalized quaternions are given by
$$JQ_{\alpha,\beta,n}^{(3)}=J_{n}^{(3)}+e_{1}J_{n+1}^{(3)}+e_{2}J_{n+2}^{(3)}+e_{3}J_{n+3}^{(3)},$$
where $J_{n}^{(3)}$ is the $n$-th third-order Jacobsthal number and $e_{1}$, $e_{2}$ and $e_{3}$ are quaternionic units which satisfy the equalities
\begin{align*}
e_{1}^{2}&=-\alpha, \ e_{2}^{2}=-\beta, \ e_{3}^{2}=-\alpha\beta,\\
e_{1}e_{2}&=e_{3}=-e_{2}e_{1}, \ e_{2}e_{3}=\beta e_{1}=-e_{3}e_{2}, \ e_{3}e_{1}=\alpha e_{2}=-e_{1}e_{3},
\end{align*}
For $\alpha=\beta=1$, the third-order Jacobsthal generalized quaternion $JQ_{1,1,n}^{(3)}$ which was given by \cite{Cer} becomes the real third-order Jacobsthal quaternions. For $\alpha=1$ and $\beta=-1$, the third-order Jacobsthal generalized quaternion $JQ_{1,-1,n}^{(3)}$ becomes the split third-order Jacobsthal quaternion. Starting from ideas given by Horadam \cite{Hor1}, Pottman and Wallner \cite{Po-Wa}, the third-order Jacobsthal generalized quaternions are studied and the relations related to these quaternions are obtained (i.e., for third-order Jacobsthal semi-quaternions, split third-order Jacobsthal semi-quaternions and third-order Jacobsthal $\frac{1}{4}$-quaternions).

\medskip
\end{document}